\documentclass[oneside, english, reqno]{amsart}
\usepackage[T1]{fontenc}
\usepackage[latin9]{inputenc}
\usepackage{amsthm}
\usepackage{amssymb}
\usepackage{esint}
\usepackage{amsmath}
\usepackage{color}
\usepackage{xcolor}

\makeatletter
\numberwithin{equation}{section}
\numberwithin{figure}{section}
\theoremstyle{plain}
\newtheorem{thm}{Theorem}[section]
\newtheorem{prop}[thm]{Proposition}
\newtheorem{definition}[thm]{Definition}

\newtheorem{lem}[thm]{Lemma}
\newtheorem{cor}[thm]{Corollary}
\newtheorem{rem}[thm]{Remark}

  \newcounter{casectr}
  
\theoremstyle{definition}
\newtheorem{assumption}{Assumption}
\theoremstyle{remark}

\makeatother

\usepackage{fancyhdr}
\usepackage{babel}
\newcommand{\AAA}{\boldsymbol{A}}

\newcommand{\eee}{\boldsymbol{f}}

\providecommand{\casename}{Case}
\newcommand*{\mathcolor}{}
\def\mathcolor#1#{\mathcoloraux{#1}}
\newcommand*{\mathcoloraux}[3]{%
  \protect\leavevmode
  \begingroup
    \color#1{#2}#3%
  \endgroup
}
\def\La{{\mathcal L}}

\begin{document}
\title{Convergence of eigenvalues to the support of the limiting measure
in  critical $\beta$ matrix models}
\author{C. Fan $^\ddagger$ }
\author{ A. Guionnet $^\dagger$}
\author{Y. Song $^\S$}
\author{A. Wang$^\sharp$}

\thanks{\quad \\$^{\ddagger} $Department of Mathematics,  Massachusetts Institute of Technology,  77 Massachusetts Ave,  Cambridge,  MA 02139-4307 USA. email:
cjfan@math.mit.edu.\\$^{\dagger}$ Department of Mathematics,  Massachusetts Institute of Technology,  77 Massachusetts Ave,  Cambridge,  MA 02139-4307 USA. email:
guionnet@math.mit.edu.\\ $^{\S}$Department of Mathematics,  Massachusetts Institute of Technology,  77 Massachusetts Ave,  Cambridge,  MA 02139-4307 USA. email:
yuqisong@mit.edu.\\ $^{\sharp}$Department of Mathematics,  Massachusetts Institute of Technology,  77 Massachusetts Ave,  Cambridge,  MA 02139-4307 USA. email:
wandi@.mit.edu.}
\maketitle
\begin{center}

\par\end{center}

\begin{abstract}
We consider the convergence of the eigenvalues to the support of  the equilibrium measure in the $\beta$ matrix models at criticality. We show a phase transition phenomenon,  namely that, with probability one,
 all eigenvalues will fall in the support of the limiting spectral measure
 when $\beta>1$,  whereas this   fails when $\beta<1$.
\end{abstract}
\section{\textbf{introduction and statement of the result}}

\subsection{Definitions and Known Results}
Let $\boldsymbol{B}$ be a subset of the real line. $\boldsymbol{B}$ can be chosen as the whole real line,  an interval,  or the union of finitely many disjoint intervals. For now,  let  $V:\boldsymbol{B}\rightarrow\mathbb{R}$
be an arbitrary function,  and let $\beta>0$ be a positive real number. In
this paper,  we consider the $\beta$ ensemble, i.e a sequence of $N$ random variables $(\lambda_{1}, \ldots, \lambda_{N})$
with law $\mu_{N, \beta}^{V;\boldsymbol{B}}$  defined as the probability measure on $\boldsymbol{B}^N$ given by 
\begin{equation}\label{defi1}
d\mu_{N, \beta}^{V;\boldsymbol{B}}(\lambda)=\frac{1}{Z_{N, \beta}^{V;\boldsymbol{B}}}\prod_{i=1}^{N}d\lambda_{i}e^{-\frac{N\beta}{2}V(\lambda_{i})}\mathbf{1}_{\boldsymbol{B}}(\lambda_{i})\prod_{1\leq i<j\leq N}\left|\lambda_{i}-\lambda_{j}\right|^{\beta}, 
\end{equation}
 where $Z_{N, \beta}^{V;\boldsymbol{B}}$ is  the partition function 
\begin{equation}\label{defi2}
Z_{N, \beta}^{V;\boldsymbol{B}}=\int_{\mathbb{R}}\cdots\int_{\mathbb{R}}\prod_{i=1}^{N}d\lambda_{i}e^{-\frac{N\beta}{2}V(\lambda_{i})}\mathbf{1}_{\boldsymbol{B}}(\lambda_{i})\prod_{1\leq i<j\leq N}\left|\lambda_{i}-\lambda_{j}\right|^{\beta}.
\end{equation}

If $\beta$ is equal to 1, 2, or 4,  $\mu_{N, \beta}^{V;\mathbb{R}}$
is the probability measure induced on the eigenvalues of $\Omega$
by the probability measure $d\Omega e^{-\frac{N\beta}{2}\mathrm{Tr}(V(\Omega))}$
on a vector space of real symmetric,  Hermitian,  and self-dual quaternionic
$N\times N$ matrices respectively,   see \cite{meh}.\\
 Therefore,  the $\beta$ models  can be viewed as the natural generalization of these matrix models and we will refer to
$\lambda_{i}$  as ``eigenvalue'' of a "matrix
model". For a quadratic potential,  the $\beta$ ensembles can also be realized as the eigenvalues of tridiagonal matrices  \cite{DE02}. Even though 
such a construction is not known for general potentials,  $\beta$ matrix models are natural Coulomb interaction probability measures which appear in many different settings.
These laws have   been intensively studied,  both in physics and in mathematics. In particular,  the convergence of the empirical measure  of the $\lambda_i$'s (which we will call hereafter the spectral measure)
 was proved
 \cite{SaffTotik, Defcours,  AGZ},  and its fluctuations analyzed \cite{Johansson98,  Pasturs,  Shch}. Moreover,  the partition functions as well as the mean Stieltjes transforms can be  expanded 
 as a function of the dimension  to all orders \cite{bp78, ACKMe, ACM92, Ake96, CE06, Ch06, BG1, BG2}.  It turns out that both central limit theorems and all order expansions depend heavily 
 on whether the limiting spectral measure has a connected support. Indeed,  when the limiting spectral measure has a disconnected support,  it turns out that even though most 
 eigenvalues will stick into one of its connected components,  some eigenvalues will randomly switch from one to the other connected components of the support even at the large dimension limit.  This phenomenon can invalidate the central limit theorem,
 see e.g.  \cite{Pasturs, Shch}, and results with the presence of a Theta function in the large dimension expansion of the partition function \cite{BG2}. 
 In the case where the limiting measure has a connected support $\boldsymbol{S}$,  and that the eigenvalues are assumed to belong asymptotically to $\boldsymbol{S}$,  even more refined
 information could be derived. Indeed,  in this case,   local fluctuations  of the $\lambda_i$'s
 could first be established in the case corresponding to Gaussian random matrices,  $\beta=1, 2$ or $4$  and $V(x)=x^2$ \cite{meh},  then
 to tridiagonal ensembles (all $\beta\ge 0$ but $V(x)=x^2$) \cite{RRV06} and more recently for general potentials and $\beta\ge 0$ \cite{BEY1, BEY2,  BEY3, Shch, BFG}. However,  all these articles consider non-critical potentials. 
 We shall below define more precisely the later case but let us say already that a non-critical potential  prevents the eigenvalues to deviate from the support of the limiting spectral measure as
 the dimension goes to infinity.
 We study in this article  $\beta$ models with critical  potentials and whether  the eigenvalues stay confined in the limiting support.
 In fact,  we exhibit an interesting phase transition: we show that if $\beta>1$ the eigenvalues stay confined whereas if $\beta<1$ some  deviate towards the critical point with  probability one.
 We postpone the study of the critical case $\beta=1$ to further research. Let us finally point out that the case where the potential is critical,  but with critical parameters tuned with the dimension so that 
 new phenomena occur,  was studied in \cite{tb, eynardbirth}. We restrict ourselves to potentials independent of the dimension. 
 
 We next describe more precisely  the definition of criticality 
 and our results.

Consider the \textit{spectral  measure} $L_{N}:=\frac{1}{N}\sum_{i=1}^{N}\delta_{\lambda_{i}}$, 
where $\delta_{\lambda_{i}}$ is the Dirac measure centered on $\lambda_{i}$. $L_N$ belongs to the set $M_1(\boldsymbol{B})$ of probability measures on the real line.
We endow this space with the weak topology. 
Then,  $L_N$ converges almost surely. This convergence can be derived from the  following large deviation result (see \cite{BAG},  and \cite[Theorem 2.6.1]{AGZ})
: \begin{thm} \label{large-deviation} Assume that $V$ is continuous and goes to infinity faster than $2\log|x|$ (if $\boldsymbol{B}$ is not bounded). The law of $L_{N}$ under $\mu_{N,\beta}^{V; \boldsymbol{B}}$ satisfies a large deviation
principle with speed $N^{2}$ and good rate function $\tilde{\mathcal{E}}$, 
where $ \tilde{\mathcal{E}}={\mathcal {E}}-\inf\{  {\mathcal {E}}(\mu), \mu\in M_1(\boldsymbol{B})\}$ with
\begin{equation}
{\mathcal{E}}[\mu]=\frac{\beta}{4} \iint \left( V(\xi)+V(\eta) -2\log\left|\xi-\eta\right|\right)d\mu(\xi)d\mu(\eta)\,.
\end{equation}
In other words,  
\begin{enumerate}
\item $$\lim_{N\rightarrow\infty}\frac{1}{N^2}\log Z_{N,\beta}^{V;\boldsymbol{B}}=-\inf_{\mu\in M_1(\boldsymbol{B})}  \mathcal E[\mu]\,.$$
\item $\tilde{\mathcal{E}}:M_{1}(\mathbb{R})\rightarrow[0, \infty]$ possesses compact
level sets $\{v:\tilde{\mathcal{E}}(v)\leq M\}$ for all $M\in\mathbb{R}^{+}$.
\item for any open set $O\subset M_{1}(\boldsymbol{B})$, 
\[
\liminf_{N\rightarrow\infty}\frac{1}{N^{2}}\log\mu_{N, \beta}^{V;\boldsymbol{B}}\left(L_{N}\in O\right)\geq-\inf_{O}\tilde{\mathcal{E}}.
\]

\item for any closed set $F\subset M_{1}(\boldsymbol{B})$, 
\[
\liminf_{N\rightarrow\infty}\frac{1}{N^{2}}\log\mu_{N, \beta}^{V;\boldsymbol{B}}\left(L_{N}\in F\right)\leq-\inf_{F}\tilde{\mathcal{E}}.
\]

\end{enumerate}

\end{thm}  The 
minimizers of $\mathcal{E}$
are described as follows  (see \cite[ Lemma 2.6.2]{AGZ}): 
\begin{thm}\label{t2} 
$\mathcal{E}$
achieves  its minimal value at a unique minimizer  $\mu_{\mathrm{eq}}$. Moreover,  $\mu_{\rm eq}$
has a compact support $\boldsymbol{S}$. In addition,  there exists a constant $C_V$ such
that: 
\begin{equation}\label{theeq}
\begin{cases}
\mathrm{for}\; x\in \boldsymbol{S} & 2\int_{\mathbb{R}}d\mu_{\mathrm{eq}}(\xi)\ln\left|x-\xi\right|-V(x)=C_V\\
\mathrm{for}\; x\;\mathrm{ Lebesgue\; almost\; everywhere\; in\;} \boldsymbol{S}^c & 2\int_{\mathbb{R}}d\mu_{\mathrm{eq}}(\xi)\ln\left|x-\xi\right|-V(x)<C_V.
\end{cases}
\end{equation}

\end{thm} 
We will refer to $\mu_{\rm eq}$, which is compactly supported,  as the equilibrium measure.\par
\begin{rem}\label{weak convergence}  Theorem  \ref{large-deviation} and Theorem  \ref{t2} imply that  under $\mu_{N, \beta}^{V;\boldsymbol{B}}$, $L_{N}$ converges  to the equilibrium measure $\mu_{\rm eq}$
almost surely.
\end{rem}\par
Once the existence of the equilibrium measure is established,  one may explore the convergence of the eigenvalues to the support of the equilibrium measure $\mu_{eq}$. It is shown in  \cite{BG1,BG2} that the probability that eigenvalues escape this limiting support is governed by a large deviation principle with rate function given by 
\begin{equation}
\mathcal{\tilde{\mathcal{J}}}^{V;\boldsymbol{B}}(x)=\mathcal{J}^{V;\boldsymbol{B}}(x)+C_V
\end{equation}
with
\begin{equation}\label{jjj}
\mathcal{J}^{V;\boldsymbol{B}}(x)=\begin{cases}
V(x)-2\int d\mu_{\mathrm{eq}}(\xi)\ln\left|x-\xi\right| & x\in\boldsymbol{B}\backslash \boldsymbol{S}\\
-C_{V} & \mathrm{otherwise}.
\end{cases}
\end{equation}
The large deviation principle states as follows:
\begin{thm} \label{original-thm} Assume $V$ continuous and going to infinity faster than $2\log|x|$ (in the case where $\boldsymbol{B}$ is not bounded). Then
\begin{enumerate}
\item $\tilde{\mathcal{J}}^{V;\boldsymbol{B}}$ 
is a  good rate function. 
\item 
We have large deviation estimates: for any $\mathsf{F} \subseteq \overline{\boldsymbol{B}\backslash \boldsymbol{S}}$ closed and $\mathsf{O} \subseteq \boldsymbol{B}\backslash \boldsymbol{S}$ open,
\begin{eqnarray*}
\limsup_{N\rightarrow\infty}\frac{1}{N}\ln \mu^{V;\mathsf{B}}_{N,\beta}\left[\exists i\quad\lambda_i \in \mathsf{F}\right] & \le & -\frac{\beta}{2}\,\inf_{x \in \mathsf{F}} \tilde{ \mathcal{J}}^{V;\boldsymbol{B}}(x), \\
\liminf_{N\rightarrow\infty}\frac{1}{N}\ln \mu^{V;\mathsf{B}}_{N,\beta}\left[\exists i\quad\lambda_i \in \mathsf{O}\right] & \ge & -\frac{\beta}{2}\,\inf_{x \in \mathsf{O}} \tilde {\mathcal{J}}^{V;\boldsymbol{B}}(x).
\end{eqnarray*}
\end{enumerate}

\end{thm} 
The last theorem shows that the support of the spectrum is governed by the minimizers of $\tilde{\mathcal{J}}^{V;\boldsymbol{B}}$ . 
\begin{definition}
Assume $V$ is continuous. We say that $V$ is {\it non-critical}
iff {$\tilde{\mathcal{J}}^{V;\boldsymbol{B}}$} is positive everywhere outside of the support of $\mu_{\rm eq}$. 
\end{definition}
\begin{rem}
Theorem \ref{t2} only ensures {$\tilde{\mathcal{J}}^{V;\boldsymbol{B}}$} is positive \textbf{almost} everywhere outside of the support of $\mu_{\rm eq}$.
\end{rem}
\begin{rem} In the literature, the potential $V$ is also said to be critical when the density of the equilibrium measure vanishes at an interior point or at an edge point at a fast rate.
We will assume that our potential is \it{not} critical in this sense (see the fourth point in Assumption \ref{main-assume}).
\end{rem}
A consequence of the second part of the aforementioned theorem is
the following: 
\begin{cor} \label{original-cor}
Let the assumptions in Theorem \ref{original-thm} hold. Assume that $V$ is non-critical. Then
\begin{equation}\label{re}
\lim_{N\rightarrow\infty}\mu_{N, \beta}^{V;\boldsymbol{B}}\left(\exists \lambda_{i}\notin \boldsymbol{A}\right)=0\,, 
\end{equation}
for any open set $\boldsymbol{A}$ containing the support of $\mu_{\rm eq}$.
\end{cor}
\begin{rem}\label{rem3} 
 Since the law of the eigenvalues  satisfies a large deviation principle with rate N,  the eigenvalues actually converge to the support exponentially fast (or more precisely, $\exists c >0, $ s.t, $\mu_{N, \beta}^{V;\boldsymbol{B}}\left(\exists \lambda_{i}\notin \boldsymbol{A}\right)\leq e^{-c N}$ for any open set $\boldsymbol{A}$ containing the support of $\mu_{\rm eq}$ ).
 \end{rem}\par 
 \begin{rem}\label{imprem}
By the definition of the partition function, $1-\mu_{N, \beta}^{V;\boldsymbol{B}}\left(\exists \lambda_{i}\notin \boldsymbol{A}\right)=\frac{Z_{N, \beta}^{V;\boldsymbol{A}}}{Z_{N, \beta}^{V;\boldsymbol{B}}}$, thus, 
\begin{equation} 
(\ref{re})\Leftrightarrow \lim_{N\rightarrow \infty}\frac{Z_{N, \beta}^{V;\boldsymbol{A}}}{Z_{N, \beta}^{V;\boldsymbol{B}}}=1.
\end{equation}
\end{rem}
In the rest of this article we investigate what happens in the case where $V$ is critical. 
This investigation will require the uses of precise estimates on $\beta$ models partitions functions derived in \cite{BG1,BG2}
and to apply these results we shall make the following assumption :
\begin{assumption} \label{main-assume} 
\begin{itemize}
\item $V:\boldsymbol{B}\rightarrow\mathbb{R}$ is a continuous function independent
of $N$. 
\item If $\pm\infty\in\boldsymbol{B}$,  
\begin{equation}
\liminf_{x\rightarrow\pm\infty}\frac{V(x)}{2\ln\left|x\right|}>1.
\end{equation}

\item $\mathrm{supp}\left(\mu_{\mathrm{eq}}\right)$ is a finite union of disjoint intervals,  i.e. %
$\mathrm{supp}\left(\mu_{\mathrm{eq}}\right)$ of the form $\boldsymbol{S}=\bigcup_{h=1}^{g}\boldsymbol{S}_{h}$, 
where $\boldsymbol{S}_{h}=[\alpha_{h}^{-}, \alpha_{h}^{+}]$.
\item Let $\boldsymbol{B}=\cup_h [b_h^-,b_h^+]$ with $b_h^-\le\alpha_h^-\le\alpha_h^+\le b_h^+$ and set $ \mathrm{Hard}=\{a\in\cup\{\alpha_h^-,\alpha_h^+\}:\alpha_h^\pm=b_h^\pm\}$ and
$\mathrm{Soft}=\cup\{\alpha_h^-,\alpha_h^+\}\backslash  \mathrm{Hard}$. 
Then we assume that 
\[
S(x)=\pi\frac{d\mu_{\mathrm{eq}}}{dx}\sqrt{\left|\frac{\prod_{a\in\mathrm{Hard}}(x-a)}{\prod_{a\in\mathrm{Soft}}(x-a)}\right|}.
\]
is strictly positive  whenever $x\in\boldsymbol{S}$.
 
\item V is a real analytic  function in some open neighborhood  $\AAA$ of $\boldsymbol{S}$ : $\AAA= \cup_{h=1}^g \boldsymbol{A}_h$, $\boldsymbol{A}_h=(a_h^-,a_h^+)$
 for some $a_h^-<\alpha_h^-<\alpha_h^+<a_h^+$, and $A_{h},A_{h'}$ are disjoint for any $h\neq h'$.

\end{itemize}
\end{assumption}
\begin{rem} When $V$ is analytic in a neighborhood of the real line, the third point of our assumption is automatically satisfied. Here, we assume analyticity only in a neighborhood of $\boldsymbol{S}$. 
\end{rem}
\begin{rem} Hereafter the neighborhood $\AAA$ will be fixed, but clearly can be chosen as small as wished, being given it is open and containing $\boldsymbol{S}$.
\end{rem}

We want to investigate whether (\ref{re}) still holds when the restriction on $\tilde{\mathcal{J}}^{V;\boldsymbol{B}}$ is  weakened so that it vanishes outside the support $\boldsymbol{S}$. Our working hypothesis
will be the following:
 \begin{assumption}\label{as}
 With $\AAA$ given in
 Assumption \ref{main-assume}, assume that 
$\tilde {\mathcal{J}}^{V;\boldsymbol{B}}$ vanishes only on the support of the equilibrium measure $\boldsymbol{S}$ and at
 one point $c_{0}$ in $\overline{\AAA}^c$. We also require  that $V$  (and therefore $ \tilde {\mathcal{J}}^{V;\boldsymbol{B}}$) extends as a twice continuously differentiable 
 function in a open neighborhood $(c_0-\varepsilon, c_0+\varepsilon)$ of $c_0$, for some $\varepsilon>0$ and that 
  $\frac{d^{2}}{dx^{2}} \tilde {\mathcal{J}}^{V;\boldsymbol{B}}(c_{0})>0$.
  Moreover, for technical reason,  we require $\frac{d^{2}}{dx^{2}} V\geq\sigma>0$ on $\boldsymbol{A}$. \end{assumption} 
\begin{rem}\label{remsup}
 Because  $\tilde{\mathcal{J}}^{V; \boldsymbol{B}}$ is a good rate function which is positive 
outside $\boldsymbol{A}\cup (c_{0}-\epsilon, c_{0}+\epsilon)$ for  any $\epsilon>0$ small enough  under our assumptions, Theorem \ref{original-thm} 
implies  that   $\mu_{N, \beta}^{V;\boldsymbol{B}}\left(\exists \lambda_{i}\notin \boldsymbol{A}\cup(c_{0}-\epsilon,c_{0}+\epsilon)\right) $ goes to zero exponentially fast. In other words,  it is bounded above  by $e^{-Nc_{\epsilon}}$ for some $c_{\epsilon}>0$.
\end{rem}{
Potentials $V$ satisfying {Assumptions  \ref{main-assume} and \ref{as}} are easy to build. Indeed, being given a probability measure $\mu$ so that $x\rightarrow f(x)=2\int\log|x-y|d\mu(y)$ is well defined  and continuous on the
whole real line,  we simply choose $V$ to be equal to $f$ on the support  $S$ of $\mu$ and Lebesgue almost surely strictly greater than $f$ outside $\boldsymbol{S}$.  This insures that $\mu$ is the
equilibrium measure of our $\beta$-model with potential $V$ as it satisfies \eqref{theeq}.  We can then choose  $V-f$   strictly positive outside $\boldsymbol{S}$ except at the point $c_0$ where it is strictly convex. To make sure that $V$ also satisfies Assumption \ref{main-assume}, we can take $\mu$ to be the equilibrium measure for a potential $W$ which is real-analytic,  going to infinity faster than $2\ln |x|$,
 strictly convex in a neighborhood of $\boldsymbol{S}$, and non-critical  (in the one cut case, we can take $W$ strictly
convex everywhere).  We let $\AAA=\cup_{h=1}^g (a_h^-,a_h^+)$ be a   bounded, open neighborhood 
 of the support of $\mu$.  We choose $V=W+C_{W}$ on $ (-\infty,a_g^+)$. We may assume without loss of generality  that $\tilde{\mathcal J}^{V;\boldsymbol{B}} (a_g^+)$ is strictly positive.  We then choose for $x>a_g^+$, $V(x)-f(x)=d(x-c_0)^2$ for some $d>0$ and $c_0>a_g^+$ so that $d(a_g^+-c_0)^2=\tilde{\mathcal J}^{W;\boldsymbol{B}}(a_g^+)$. 
 We have constructed an equilibrium measure $\mu$ for a potential $V$ satisfying Assumptions  \ref{main-assume} and \ref{as}.  }

\subsection{Main Results}
\begin{thm}\label{main}
Given Assumptions \ref{main-assume} and  \ref{as},   and  with $\AAA$ as in assumption \ref{main-assume},  $c_0,\varepsilon$ as in Assumption \ref{as}, we have the following alternative :
\begin{itemize}
\item when $\beta>1$, 
\begin{equation}
\lim_{N\rightarrow\infty}\mu_{N, \beta}^{V;\boldsymbol{B}}\left(\exists \lambda_{i}\notin \boldsymbol{A}\right)=0,  \end{equation}
\item when $\beta<1$, 
\begin{equation}
\lim_{N\rightarrow\infty}\mu_{N, \beta}^{V;\boldsymbol{B}}\left(\exists \lambda_{i}\notin \boldsymbol{A}\right)=1.
\end{equation} 
\end{itemize}
{Equivalently, for  any $\epsilon\in (0,\varepsilon)$,
the probability that there exists an eigenvalue in $(c_0-\epsilon,c_0+\epsilon)$ goes to zero when $\beta>1$  and to one when $\beta<1$. }
\end{thm}
The behavior below $\beta=1$ can be illustrated with the case $\beta=0$ where one would consider a potential $V$ vanishing on a support $\boldsymbol{S}$ and at a point $c_0$ (where its second derivative is positive), being strictly positive everywhere else. This corresponds to $N$ independent variables with probability of order $N^{-1/2}$ to belong to a small neighborhood of $c_0$ (where the latter probability can be estimated by Laplace method). However, we will
show in this article that
 this weight has to be corrected by a term $N^{-\frac{\beta}{2}}$ by studying the precise estimates derived in \cite{BG1,BG2}.  Essentially, digging back into the latter estimates, one can realize that these corrections 
come from Selberg integral and in fact is due to the Coulomb gas interaction. In this case, it is clear 
that some eigenvalues will lie in the neighborhood of $c_0$ with positive probability as soon as $\beta+1<2$. The existence of a phase transition for this phenomenon 
at $\beta=1$ is new  to our knowledge.  It suggests that the support of the eigenvalues of matrices with real coefficients corresponding to $\beta=1$ matrix models
might be  more sensible to perturbations of the potential than matrices with complex coefficients (corresponding to $\beta=2$). This is however not supported by finite dimensional perturbations of the Wigner matrices since the BBP transition \cite{BBP} does not vary much between these two cases. 
Let us observe that our arguments could be carried similarly with several critical points similar to $c_0$ without changing the phase transition. However, if
the second derivative of $ {\mathcal{J}^{V;\boldsymbol{B}}}$ at these critical points could vanish so that $ {\mathcal{J}}^{V;\boldsymbol{B}}$ behaves as $|x-c_0|^q$ in the vicinity of $c_0$ for some $q> 2$,
the phase transition would occur at a threshold $\beta_q$ depending on $q$ (see  Remark \ref{beha}). 
\subsection{Structure of the paper}
In Section 2 we reduce the problem to the analysis of  the probability that  M eigenvalues are contained in a small neighborhood of $c_{0}$ while the rest of the $N-M$  eigenvalues are contained in $\boldsymbol{A}$ and state the main proposition, Proposition \ref{mainprop}, which give precise estimates of this probability.
We deduce our main result Theorem \ref{main} in the case $\beta>1$ in   Section \ref{betamore}  and the case $\beta<1$ in  Section \ref{betaless}. 
Section \ref{proofprop1} is devoted to the proof of Proposition \ref{mainprop},  which we first give in the case where the equilibrium measure has a connected support and then extend to the general case. The appendix contains precise concentration of measures results which are key to our estimates.
\subsection{Notation}
We use the  notation %
$X\lesssim Y$ (resp. %
$X\gtrsim Y$) to denote $X\leq CY $ (resp. %
$X\geq CY$) for some universal constant $C$ . %And 
$X\approx Y$ when both 
$X\lesssim Y$ and $X\gtrsim Y$ hold. We sometimes
use $a\ll 1$ to denote that $a$ is smaller than any universal constant involved in the proof.
\section{\textbf{Preliminary and Basic analysis}}
The probability that a specific subset of $M$ eigenvalues are contained in a small neighborhood $(c_0-\epsilon,c_0+\epsilon)$ of $c_0$ while the other $N-M$  eigenvalues are contained in $\boldsymbol{A}$ shall be denoted by  $\Phi_{N, M, \beta}^{V;{\epsilon}}$ :
 \begin{equation}\label{main3}
\Phi_{N, M, \beta}^{V;{\epsilon}}:=\mu_{N, \beta}^{V;\boldsymbol{B}}\left(\lambda_{N-M+1}, ..., \lambda_{N}\in (c_{0}-\epsilon, c_{0}+\epsilon), \lambda_{1}, ...\lambda_{N-M}\in\boldsymbol{A}\right).
\end{equation} {{
$\Phi_{N, M, \beta}^{V;{\epsilon}}$ depends also  on $\AAA$ and $\boldsymbol{B}$   but it  will be fixed hereafter as in Assumption \ref{main-assume} so that we do not stress this dependency. $\epsilon$ will later be chosen small enough, but notice  that the conclusion will not depend on this choice since the probability that eigenvalues go in $[c_0-\varepsilon,c_0+\varepsilon]\backslash [c_0-\epsilon,c_0+\epsilon]$ goes to zero as $N$ goes to infinity for any $\epsilon>0$ by the large deviation principle Theorem \ref{original-thm}  and Assumption \ref{as}.}}
The key to prove our main result is to compute
 the speed at which $\Phi_{N, M, \beta}^{V;{\epsilon}}$  goes to zero  as $N$ goes to +$\infty$. Indeed, 
note that
\begin{eqnarray}
\quad\mu_{N, \beta}^{V;\boldsymbol{B}}\left(\exists \lambda_{i} \notin \boldsymbol{A}\right)
&=& \mu_{N, \beta}^{V;\boldsymbol{B}}\left(\exists \lambda_{i} \notin \boldsymbol{A}\cup (c_{0}-\epsilon, c_{0}+\epsilon)\right) \nonumber\\
&+&\sum_{M>\delta N} \left(\begin{array}{c}
N\\
M
\end{array}\right){\Phi}^{V;{\epsilon}}_{N, M, \beta}\nonumber\\
&+&\sum_{1\le M \leq \delta N}\left(\begin{array}{c}
N\\
M
\end{array}\right){\Phi}^{V;{\epsilon}}_{N, M, \beta}\nonumber\\
&=:&P_{1}+P_{2}+P_{3}.\label{finalstep}
\end{eqnarray}
Here $\delta>0$ is a small fixed constant {which will be chosen later}.\par
Since $\tilde{\mathcal{J}}^{V; \boldsymbol{B}}$ is a good rate function which is positive 
outside $\boldsymbol{A}\cup [c_{0}-\epsilon, c_{0}+\epsilon]$, Theorem \ref{original-thm} 
implies  that for any fixed $\epsilon>0$, $P_{1}$ approaches 0 exponentially fast. In other words,  it is controlled by $e^{-Nc_{\epsilon}}$ for some $c_{\epsilon}>0$.\par
$P_2$ is bounded above by the probability that there are at least $\delta N$ eigenvalues close to $c_0$. This implies that the empirical measure $L_{N}$ must put a mass $\delta$ in $(c_0-\epsilon,c_0+\epsilon)$, so that it must be at a positive 
distance of $\mu_{\rm eq}$ since  $(c_0-\epsilon,c_0+\epsilon)\cap\boldsymbol{S}=\emptyset$. 
By the large deviation principle for the law of $L_{N}$ described in Theorem \ref{large-deviation}, $P_{2}$ is bounded above  by $e^{-{c_{\delta}N^{2}}}$ for some $c_{\delta}>0$.  Therefore we deduce that
for any $\delta,\varepsilon >0$, there exists $c(\delta,\epsilon)>0$ such that 
\begin{equation}\label{finalstep2}
\mu_{N, \beta}^{V;\boldsymbol{B}}\left(\exists \lambda_{i} \notin \boldsymbol{A}\right)= P_{3} +O(e^{-c(\delta,\epsilon)N}).
\end{equation}
Our goal,  therefore,  is to control the third term $P_{3}$.\par
 Since  $\mathcal{J}^{V;\boldsymbol{B}}$ goes to infinity at infinity, Theorem \ref{original-thm} also shows that the probability to have an eigenvalue above some finite threshold  goes to zero
 exponentially fast. Therefore, we may assume without loss of generality that $\boldsymbol{B}$ is a bounded set.\par
Thus,  we are left to analyze ${\Phi}^{V;{\epsilon}}_{N, M, \beta}$ for a bounded set $\boldsymbol{B}$.\par
We prove the following  bounds in section \ref{proofprop}:
\begin{prop}\label{mainprop}
\label{multi-eig-thm} Let Assumptions \ref{main-assume} and \ref{as} hold. Then, there exist $c, \delta_0,\epsilon_0>0$ so that  for $\delta\in (0,\delta_0)$,  $\epsilon\in (0,\epsilon_0\wedge \varepsilon)$, 
 we have uniformly in $M\le \delta N$
 \begin{equation}\label{ub}
\Phi_{N, M, \beta}^{V;{\epsilon}}\lesssim \frac{1}{N^{\frac{M(\beta+1)}{2}}} +O(e^{-c N^2}).
\end{equation}
On the other hand, 
\begin{equation}\label{lb}
\frac{1}{N^{\frac{(\beta+1)}{2}}}\frac{Z_{N,\beta}^{V; \boldsymbol{ A}}}{Z_{N,\beta}^{V;\boldsymbol{B}}}\lesssim  \Phi_{N, 1, \beta}^{V;{\epsilon}}+{O(e^{-c N^{2}})}.
\end{equation}

\end{prop}

{If one assumes  $\boldsymbol{A}$ is connected, Proposition \ref{mainprop} follows from the calculation from \cite{BG1} where the one-cut case in considered.  In the general multi-cut case,
the proof of Proposition \ref{mainprop} is based on the precise estimate derived in  \cite{BG2} for the partition function and correlators for fixed filling fraction measure, that is with given number of eigenvalues in each connected part of the support $\boldsymbol{S}$.  One could also consider the more general multi-cut  case where the potential $V$ is replaced by a non-linear statistic and then use similar estimates derived in \cite{BGK}.}

 We next give the proof of our main result.

\section{Convergence  of the eigenvalues to the support $\boldsymbol{S}$ when $\beta>1$}\label{betamore}
We next prove  the first half of our main Theorem \ref{main}. To this end, we use the upper bound \eqref{ub} provided by  Proposition \ref{mainprop} to find
\begin{equation}
P_{3}\leq \sum_{1\le M\le \delta N}\left(\begin{array}{c}
N\\
M
\end{array}\right)\Phi_{N,M, \beta}^{V;{\epsilon}} \lesssim \left( 1+N^{-\frac{(\beta+1)}{2}}\right)^N-1\,.
\end{equation}
where we used  that any error of order $e^{-cN^2}$ in $\Phi_{N, M, \beta}^{V;{\epsilon}}, 1\le M\le \delta N,$ is neglectable in the above sum. Hence,  when $\beta>1$, $P_{3}$ goes to $0$ as $N$ goes to $+\infty$.  We deduce by  \eqref{finalstep2} that 
$$\lim_{N\to\infty} \mu_{N, \beta}^{V;\boldsymbol{B}}\left(\exists \lambda_{i} \notin \boldsymbol{A}\right)=0\,.$$
Moreover,  this is equivalent to  the fact that $ \mu_{N, \beta}^{V;\boldsymbol{B}}\left(\exists \lambda_{i} \in (c_0-\epsilon, c_0+\epsilon)\right)$
goes to zero by Remark \ref{remsup}. 
 
\section{ Escaping eigenvalues when $\beta<1$}\label{betaless}
We prove that when $\beta<1$ the probability that no eigenvalues
lies in the neighborhood of $c_0$ goes to zero, that is  by Remark \ref{imprem}, that we have:
\begin{equation}\label{eod}
\lim_{N\rightarrow \infty}\frac{Z_{N, \beta}^{V;\boldsymbol{A}}}{Z_{N, \beta}^{V; \boldsymbol{B}}}=0.
\end{equation}
This is done by lower bounding the probability $p_{N,\beta}^{V}$
 that one eigenvalue exactly lies in the neighborhood of
$c_0$. Indeed, as $A_i=\{\lambda_i\in (c_0-\epsilon,c_0+\epsilon), \lambda_j\in \boldsymbol{A}, j\neq i\}$ are disjoint as soon as $(c_0-\epsilon,c_0+\epsilon)\cap \boldsymbol{ A}=\emptyset$,
we have by symmetry
$$p_{N,\beta}^V=N\Phi_{N, 1, \beta}^{V;{\epsilon}}.$$
Since $p_{N,\beta}^V\le 1$,  we deduce  from \eqref{lb} that
$$\frac{N}{N^{\frac{(\beta+1)}{2}}}\frac{Z_{N,\beta}^{V; \boldsymbol{A}}}{Z_{N,\beta}^{V;\boldsymbol{B}}}
\lesssim 1+N\times {O(e^{-c N^{2}})}\,,$$ which results with
$$\frac{Z_{N,\beta}^{V; \boldsymbol{A}}}{Z_{N,\beta}^{V;\boldsymbol{B}}}\lesssim  N^{\frac{\beta-1}{2}}\,,$$
so that \eqref{eod}, and therefore the second part  of our main Theorem \ref{main},
follows.

We first prove  Proposition \ref{mainprop} in the case where $\mu_{\rm eq}$ has a connected support (the one cut case) where
our proof is based on the expansion of partition functions obtained in the one cut case in \cite{BG1}, 
and then turn to the more delicate general  case (multi-cut case) which is based from estimates from \cite{BG2}.

\section{{Proof of   Proposition \ref{mainprop} }}\label{proofprop1}
We start with the explicit formula for $\Phi_{N, M, \beta}^{V; {\epsilon}}$ which can be written as follows :
\begin{eqnarray}
\Phi_{N,M,\beta}^{V;{\epsilon}}
&=&\frac{Z_{N-M,\beta}^{\frac{N}{N-M} V; \boldsymbol{A}}}{Z_{N, \beta}^{V;\boldsymbol{B}}}
\int_{[c_0-\epsilon,c_0+\epsilon]^M}   \Xi(\eta_1,\ldots, \eta_M) \prod_{1\le k<l\le M} |{\eta_k-\eta_l}|^\beta
\prod_{j=1}^M e^{-\frac{\beta M}{2} V(\eta_j)} d\eta_j \nonumber
\end{eqnarray}
where

$${\Xi(\eta_1, \cdots, \eta_M):=\mu^{\frac{N}{N-M}V;\boldsymbol{A}}_{N-M, \beta}\left(\prod_{j=1}^{M}e^{\beta\sum_{i=1}^{N-M}\ln|\eta_{j}-\lambda_{i}|-\frac{\beta}{2}(N-M)V(\eta_{j})}\right)}$$
The term $\prod_{1\leq k<l\leq M}|\eta_{k}-\eta_{l}|^{\beta}$ is  bounded by $(2\epsilon)^{\frac{\beta M(M-1)}{2}}$.
Thus,  we have:
\begin{equation}\label{main4}
\Phi_{N, M, \beta}^{V;{\epsilon}}\leq (2\epsilon)^{\frac{\beta M(M-1)}{2}}
Y_{N, M} L_{N,M}\,.
\end{equation}
with \begin{eqnarray}\label{defLY} 
L_{N,M}&:=&\int_{[c_{0}-\epsilon, c_{0}+\epsilon]^{M}}\Xi(\eta_{1}, \cdots, \eta_{M})\prod_{j=1}^{M}e^{-\frac{M\beta}{2}V(\eta_{j})}d\eta_{j}.\nonumber\\
Y_{N, M}&:=&\frac{Z^{\frac{N}{N-M}V;\boldsymbol{A}}_{N-M, \beta}}{Z^{V;\boldsymbol{B}}_{N, \beta}},\nonumber\\
\end{eqnarray}
We wish to split $Y_{N, M}$ into components for further analysis. We make the decomposition:
\begin{equation}\label{main5}
Y_{N, M}=\frac{Z^{V,\boldsymbol{A}}_{N,\beta}}{Z^{V,\boldsymbol{B}}_{N,\beta}}
 \tilde{Y}_{N, M}\,, \qquad 
 \tilde{Y}_{N, M}
=F_{N,M}G_{N,M}\,,
\end{equation}
where 
 \begin{equation}\label{defFG}
F_{N, M}=\frac{Z_{N-M, \beta}^{\frac{N}{N-M}V;\boldsymbol{A}}}{Z_{N-M, \beta}^{V;\boldsymbol{A}}}\,,\qquad G_{N, M}=\frac{Z_{N-M, \beta}^{V;\boldsymbol{A}}}{Z_{N, \beta}^{V;\boldsymbol{A}}}\,.\end{equation}
{To get the upper bound in Proposition \ref{mainprop}, we will use the trivial estimate $Y_{N,M}\leq \tilde{Y}_{N,M}$ since $Z^{V,\boldsymbol{A}}_{N,\beta}\le Z^{V,\boldsymbol{B}}_{N,\beta}$.}
 Thus one finally rewrites formula (\ref{main4}) as:
\begin{equation}\label{clean}
\Phi_{N, M, \beta}^{V;{\epsilon}}\leq (2\epsilon)^{\frac{\beta M(M-1)}{2}}
G_{N, M}F_{N, M} L_{N,M}\,.
\end{equation}
We remark here that when M=1,  the inequality (\ref{main4}) becomes an equality:
\begin{equation}\label{counter}
\Phi_{N, 1, \beta}^{V; {\epsilon}}=\frac{Z^{V;\boldsymbol{A}}_{N,\beta}}{Z^{V;\boldsymbol{B}}_{N,\beta}}
G_{N, 1}F_{N,1} L_{N,1}\,.
\end{equation}
($\ref{counter}$) will be used to prove \eqref{lb}. \\

We next estimate $G_{N,M}$  and $F_{N,M}\times L_{N,M}$.
We shall prove that
\begin{prop}\label{order0} Under Assumptions \ref{main-assume} and \ref{as},
there exists  a small $\delta>0$, such that uniformly for  $M\le \delta N$, we have
\begin{equation} 
G_{N, M}\approx C_{M}\frac{1}{N^{\frac{M\beta}{2}}}e^{NM(\frac{\beta}{2}\inf_{\xi\in\boldsymbol{B}}\mathcal{J}^{V;\boldsymbol{B}}(\xi)+\frac{\beta}{2}\int V(\eta)d\mu_{\mathrm{eq}}(\eta))}(\frac{N-M}{N})^{\frac{\beta}{2}(N-M)} \,.
\end{equation}
Furthermore,  there exists a finite  positive constant $C$ such that for all $1\le M\le \delta N$,  
\begin{equation}
\frac{1}{C}e^{-CM^{2}}\le C_{M}\le C e^{CM^{2}}.
\end{equation}
\end{prop}
For  the term $F_{N,M}\times  L_{N,M}$ we have the estimate
\begin{prop}\label{o3} Under Assumptions \ref{main-assume} and  \ref{as},  there exists a positive constant $c$ and a finite constant $C$ 
so that   for $\delta>0$ small enough,  such that uniformly on $M\le \delta N$, 
$$F_{N,M} L_{N,M}\lesssim e^{CM^2} e^{-\frac{\beta}{2}NM(\int V(\eta)d\mu_{\rm eq}(\eta)+\inf \mathcal J^{V;\boldsymbol{B}})} \frac{1}{N^{\frac{M}{2}}} +O(e^{-c N^2})$$
and
$$F_{N,1} L_{N,1}\gtrsim e^{-\frac{\beta}{2}N(\int V(\eta)d\mu_{\rm eq}(\eta)+ \inf \mathcal J^{V;\boldsymbol{B}})} \frac{1}{N^{\frac{1}{2}}}+O(e^{-c N^2})\,.$$

\end{prop}
Clearly, Propositions \ref{order0} and \ref{o3} give Proposition \ref{mainprop}. First, as the constant $c$ in Proposition \ref{o3}  is independent from $\delta$,
we can choose $\delta$ small enough so that $O(e^{-c N^2})$ is negligible.  Indeed, since  $G_{N,M}^{-1}$ is at most of order $e^{C\delta N^2}$,
 for $\delta\le c/2C$
$$G_{N,M}  e^{-c N^2}\le e^{-c N^2/2}$$
is negligible compared to 
$$G_{N,M} e^{CM^2} e^{-\frac{\beta}{2}NM(\int V(\eta)d\mu_{\rm eq}(\eta)+\inf \mathcal J^{V;\boldsymbol{B}})} \frac{1}{N^{\frac{M}{2}}} $$
which is decaying only polynomially by Proposition \ref{order0}.
We get \eqref{ub} by using \eqref{clean},  Propositions \ref{order0} and \ref{o3}, we  choose 
 $\epsilon$ small enough so that $2\epsilon e^{4C}\le 1$ so that the terms in $e^{CM^2}$ disappear, 
 and we observe that $(N-M)/N\le 1$. To derive 
  the lower bound \eqref{lb}, we use  \eqref{counter} together with 
 Propositions \ref{order0} and \ref{o3}. in this case the term $((N-1)/N)^{\frac{\beta}{2} (N-1)}$ is of order one.

 The proof of Proposition \ref{o3}  is based on the following proposition.
\begin{prop}\label{o2} Under Assumption \ref{main-assume} and \ref{as},
there exists   positive constants $c, C, \delta_0$  so that for   $\delta\in (0,\delta_0)$,  $M\le \delta N$, for any 
$\eta_{1}, \cdots, \eta_{M}$ belonging to  $[c_{0}-\varepsilon,c_0+\varepsilon]$,  we have the following uniform estimate:
\begin{equation}
F_{N, M}\Xi\left(\eta_{1}, \cdots, \eta_{M}\right)\lesssim e^{CM^{2}}e^{-\frac{\beta NM}{2}\int V(\eta)d\mu_{eq}(\eta)}e^{-\frac{\beta}{2}N\sum_{j=1}^{M}\mathcal{J}^{V;\boldsymbol{B}}(\eta_{j})}+O(e^{-c N^2}).
\end{equation}
Moreover,  for $\eta\in [c_{0}-\varepsilon,c_0+\varepsilon]$
\begin{equation}\label{lkj}F_{N, 1}\Xi\left(\eta\right)\gtrsim e^{-\frac{\beta N}{2}( \int V(x)d\mu_{eq}(x)+\mathcal{J}^{V;\boldsymbol{B}}(\eta))}+O(e^{-c N^2}).\end{equation}
\end{prop}
Let us first deduce Proposition \ref{o3} from Proposition \ref{o2}.
The proof is straightforward  since by Proposition \ref{o2}
\begin{eqnarray*}
F_{N,M}L_{N,M}&\lesssim & e^{CM^2} e^{-\frac{\beta }{2}NM\int V(\eta)d\mu_{\rm eq}(\eta) -NM\frac{\beta}{2}\inf \mathcal J^{V;\boldsymbol{B}} }\left(\int_{c_0-\epsilon}^{c_0+\epsilon}e^{-N \tilde{\mathcal J}^{V;\boldsymbol{B}}(x)}  dx\right)^M\\
&&+O(e^{-c N^2})\end{eqnarray*}
where we can use 
Laplace method (recall we assume $\tilde{\mathcal{J}}^{V;\boldsymbol{B}}(c_0)=0, \frac{d}{dx} \tilde{\mathcal{J}}^{V;\boldsymbol{B}}(c_0)=0,\frac{d^{2}}{dx^{2}} \tilde{\mathcal{J}}^{V;\boldsymbol{B}}(c_0)>0$,  see  \cite[section 3.5.3]{AGZ} for details) to get
$$\int_{c_0-\epsilon}^{c_0+\epsilon} e^{-N\tilde {\mathcal J}^{V;\boldsymbol{B}}(\lambda)}d\lambda\approx  \frac{1}{\sqrt{N}}\,.$$
The proof of the lower bound is similarly deduced from \eqref{lkj}. 
\begin{rem} \label{beha}Note that if we would have assumed instead of $\frac{d^{2}}{dx^{2}} \tilde{\mathcal{J}}^{V;\boldsymbol{B}}(c_0)>0$ that for some $q>2$, $\tilde{\mathcal{J}}^{V;\boldsymbol{B}}(x)\simeq |x-c_0|^q$
in a neighborhood of $c_0$, we would have obtained
$$F_{N,M} L_{N,M}\lesssim Ce^{CM^2} e^{-\frac{\beta}{2}NM\int V(\eta)d\mu_{\rm eq}(\eta)} \frac{1}{N^{\frac{M}{q}}} +O(e^{-c N^2})$$
and criticality would have occurred at $\beta=2/q$.
\end{rem}

\subsection{Proof of Proposition \ref{order0} and \ref{o2} in the one cut case.}
To estimate $G_{N,M}$ and $F_{N,M}L_{N,M}$ in the one cut case, we shall rely on the following results from \cite{BG1}:
\begin{thm}\label{art0}\cite[Propositions 1.1 and  1.2]{BG1}
If $V$ satisfies Assumption \ref{main-assume} and \ref{as} on $\boldsymbol{A}$,  there exists a universal constant $e$, and constants $F_\beta^{\{k\}}$ 
 so that we have for $N$ large enough:
\begin{equation}\label{expZ}
 Z_{N,  \beta}^{V;\mathsf{A}} = N^{(\frac{\beta}{2})N + e}\exp\Big(\sum_{k =-2}^K N^{-k}\, F^{\{k\}}_{ \beta} + o(N^{-K})\Big)\,.
\end{equation}
 Moreover, define the \textit{correlators}  given for $x\in \mathbb C\backslash \boldsymbol{A}$ by: 
 \begin{equation}
    W_{1} (x):=\mu_{N,  \beta}^{V;\boldsymbol{A}}(\sum \frac{1}{x-\lambda_{i}}), \quad
     W_1^{\{-1\}}(x):={\mu_{\mathrm{eq}}}(\frac{1}{x-\lambda}).
    \end{equation}
Then
\begin{equation}
\label{expco0}W_{1}(x) =  NW_1^{\{-1\}}(x) +O(1).
\end{equation}
 \eqref{expco0} holds uniformly for $x$ in compact regions outside $\boldsymbol{A}$(in particular near the critical point $c_{0}$).
\end{thm}
\noindent
{\bf Proof of Proposition \ref{order0}.}
The fact that 
\begin{equation}\label{GNM0}
G_{N, M}\approx C_{M}e^{MNA_{\beta}}\frac{1}{N^{\frac{M\beta}{2}}}(\frac{N-M}{N})^{\frac{\beta}{2}(N-M)}, 
\end{equation}
where $A_{\beta}=-2F^{\{-2\}}_\beta$ does not depend on $N$ or $M$ and $C_{M}\lesssim e^{CM^{2}}$ is a direct consequence of \eqref{expZ}.
Moreover,  the large deviation principle of Theorem \ref{large-deviation}(1) yields 
$$-F^{\{-2\}}_\beta=\inf \mathcal E= \frac{\beta}{2}( \int V(x)d\mu_{\rm eq}(x)-\int\int\log|x-y|d\mu_{\rm eq}(x)d\mu_{\rm eq}(y))\,.$$
Also,  by definition, since the effective potential is constant on the support of the equilibrium measure,
$$\inf \mathcal J^{V;\boldsymbol{B}}=-C_V= \int V(x)d\mu_{\rm eq}(x)-2\int\int\log|x-y|d\mu_{\rm eq}(x)d\mu_{\rm eq}(y)\,.$$
As a consequence
$$A_{\beta}=\frac{\beta}{2}\inf_{\xi\in\boldsymbol{A}}\mathcal{J}^{V;\boldsymbol{A}}(\xi)+\frac{\beta}{2}\int V(\eta)d\mu_{\mathrm{eq}}(\eta)\,.$$

\noindent
{\bf Proof of Proposition \ref{o2}.} In the one-cut case we can prove this proposition without the error terms $O(e^{-c N^2})$ which we will need to deal with the several cut case.
The main tool is to 
use concentration of measure. In fact we can write
$$F_{N,M}\Xi(\eta_1,\ldots,\eta_M)=\mu_{N-M,\beta}^{  V; \boldsymbol{A}}\left( e^{\sum_{i=1}^{N-M} h_\eta (\lambda_i)}\right)$$
with 
$$h_\eta (x)=\beta \sum_{i=1}^M (\ln(\eta_i-x)-\frac{1}{2} V(\eta_i)-\frac{1}{2} V(x))\,.$$
Eventhough this function depends on $M$,
concentration inequalities will allow a uniform control. We develop the necessary estimates in the appendix, 
see Lemma \ref{central}. $\|f\|_{\mathcal L}$ denotes the Lipschitz norm and $\|f\|_\infty$ the uniform bound on a neighborhood of $\boldsymbol{A}$.
 We apply Lemma \ref{central} with $h=h_\eta$. As the $\eta_i$ are away from $\boldsymbol{A}$, {$\|h_\eta\|_{\mathcal L}^2$} is uniformly bounded by $CM^2$
for some finite constant $C$, whereas $\|h_\eta\|_\infty$ is of order $M$.
Hence, we deduce from Lemma \ref{central} that
\begin{equation}\label{polk} e^{-CM} e^{(N-M)\mu_{\rm eq }(h_\eta)}\le F_{N,M} \Xi(\eta_1,\ldots,\eta_M)\le e^{C M^2} e^{ (N-M)  \mu_{\rm eq}( h_\eta) }\,.\end{equation}
Note that we can replace above $(N-M)\mu_{\rm eq }(h_\eta)$ by $N\mu_{\rm eq }(h_\eta)$ up to an error of order $ M^2$ which amounts to change the constant $C$.
This  completes the proof of  Proposition \ref{o2} since
\begin{eqnarray}
\mu_{\rm eq}( h_\eta)&=&\frac{\beta}{2}( \sum_{i=1}^M  2 \int \log |\eta_i-x|d\mu_{\rm eq}(x)-V(\eta_i))-\frac{\beta}{2} M \int V(x)d\mu_{\rm eq}(x)\nonumber\\
&=&-\frac{\beta}{2}\sum_{i=1}^M \mathcal J^{V;\boldsymbol{B}}(\eta_i) -\frac{\beta}{2} M \int V(x)d\mu_{\rm eq}(x)\,.\label{eqr}\end{eqnarray}

\subsection{{Proof of   Propositions \ref{order0} and \ref{o2} in the general multi-cut case}}\label{proofprop}

In the multi cut case, we have to be more careful since the number of eigenvalues that are in each connected component of the support of $\mu_{\rm eq}$
is not a priori fixed. The idea is therefore that we will have to sum over all possible number of eigenvalues in these components. 
Our proof is based on estimates from \cite{BG2} 
on the fixed filling fraction measure  {(see Definition \ref{defi333} below)}. We introduce some new notation below.\par
{From the viewpoint of Large Deviation Principle on $L_{N}$, the number of eigenvalues in the interval $A_{h}$, which we will call the filling fractions,  should in principle be proportionnal  to $\mu_{eq}(S_{h})$, which is the mass of equilibrium measure accumulated on $A_{h}$.}\par
{Thus, let's define $\boldsymbol{f}_{\star}$ as  the $g$-tuple denoting the mass of the equilibrium measure in each of the intervals that comprise the support of the equilibrium measure, i.e.
\begin{equation} \label{eps} \boldsymbol{f}_{\star}:=\left(\mu_{\mathrm{eq}}(S_{1}), \cdots, \mu_{\mathrm{eq}}(S_{g})\right)
\end{equation}
}
{
   In order to describe how the $N$ eigenvalues are distributed in the $g$ intervals $\AAA_{h}$, we let
\begin{equation}\label{gtuple}
   \mathcal{E}_g:=\left\lbrace(f_{1},  \cdots , f_{g})|\sum_{h=1}^{g} f_{h}=1, f_{1},  \ldots , f_{g}\geq 0\right\rbrace .
\end{equation}}\par
  Now we can define the fixed filling fraction probability measure:
\begin{definition}\label{defi333}
For any $\overrightarrow{\boldsymbol{N}}=(N_1,\ldots,N_g)$ so that $\overrightarrow{\boldsymbol{N}}/N\in \mathcal{E}_g$, let the fixed filling fraction probability measure $d\mu_{N, \frac{\overrightarrow{\boldsymbol{N}}}{N}, \beta}^{V;\boldsymbol{A}}$ be given by:
\begin{equation}\label{defi3}
\begin{aligned}
d\mu_{N, \frac{\overrightarrow{\boldsymbol{N}}}{N}, \beta}^{V;\boldsymbol{A}}(\boldsymbol{\lambda}) 
 := & \frac{1}{Z_{N, \boldsymbol{f}, \beta}^{V;\boldsymbol{A}}}\prod_{h = 1}^g \Big[\prod_{i = 1}^{N_h} \mathrm{d}\lambda_{h, i}\, \mathbf{1}_{\boldsymbol{A}_{h}}(\lambda_{h, i})\, e^{-\frac{\beta N}{2}\, V(\lambda_{h, i})}\, \prod_{1 \leq i < j \leq N} |\lambda_{h, i} - \lambda_{h, j}|^{\beta}\Big] \nonumber \\
&  \times \prod_{1 \leq h < h' \leq g} \prod_{\substack{1 \leq i \leq N_h \\ 1 \leq j \leq N_{h'}}} |\lambda_{h, i} - \lambda_{h', j}|^{\beta}, 
\end{aligned}
\end{equation}
where $Z_{N, \boldsymbol{f}, \beta}^{V;\boldsymbol{A}}$ is the partition function. 
\end{definition}
  The following precise 
estimate of the fixed filling fraction measure from \cite[Theorem 1.4]{BG2} will be essential in our proof. It extends Theorem \ref{art0}. 

\begin{thm}\label{part}
If $V$ satisfies Assumption \ref{main-assume} and \ref{as} on $\boldsymbol{A}$,  there exists $t > 0$ such that,  uniformly for $\frac{\overrightarrow{\boldsymbol{N}}}{N} \in \mathcal{E}_g$ that satisfies $|\frac{\overrightarrow{\boldsymbol{N}}}{N} - \boldsymbol{f}_{\star}| < t$,  we have:
\begin{equation}
\label{sqiqq}\frac{N!}{\prod_{h = 1}^{g} (N_h)!}\, Z_{N, \frac{\overrightarrow{\boldsymbol{N}}}{N}, \beta}^{V;\mathsf{A}} = N^{(\frac{\beta}{2})N + e}\exp\Big(\sum_{k =-2}^K N^{-k}\, F^{\{k\}}_{\frac{\overrightarrow{\boldsymbol{N}}}{N}, \beta} + o(N^{-K})\Big)\,.
\end{equation}
$e$ is some universal constant,  $F^{\{k\}}_{\boldsymbol{f}, \beta}$ extends as a smooth function for $\boldsymbol{f}$ close enough to $\boldsymbol{f}_{\star}$,  and at the value $\boldsymbol{f} = \boldsymbol{f}_{\star}$,  the derivative of $F^{\{-2\}}_{\boldsymbol{f}, \beta}$ vanishes and its Hessian is negative definite. 
 Assume that $\overrightarrow{\boldsymbol{N}}/N$ converges towards $\boldsymbol{f}$. Then, the law of the empirical measure
 $L_N$ under  $\mu_{N, \boldsymbol{f}, \beta}^{V;\boldsymbol{A}}$ satisfies a large deviation principle with speed $N^2$ and
 good rate function $\tilde{\mathcal E}_{\boldsymbol{f}}$ which is minimized at a unique probability measure $\mu_{{\rm eq}, \boldsymbol{f}}$, which is also the minimum of $\mathcal E$ under the constraint that $\mu(\boldsymbol{A}_h)=f_h$.
 In particular $L_N$ converges $\mu_{N, \frac{\overrightarrow{\boldsymbol{N}}}{N}, \beta}^{V;\boldsymbol{A}}$ almost surely to $\mu_{{\rm eq},\boldsymbol{f}}$.
 Moreover,  for $x\in \mathbb C\backslash \boldsymbol{A}$ let 
 \begin{equation}
    W_{\frac{\overrightarrow{\boldsymbol{N}}}{N}}(x):=\mu_{N, \frac{\overrightarrow{\boldsymbol{N}}}{N}, \beta}^{V;\boldsymbol{A}}(\sum \frac{1}{x-\lambda_{i}}), \quad
     W_{\boldsymbol{f}}^{\{-1\}}(x):=\mu_{\mathrm{eq}, \boldsymbol{f}}(\frac{1}{x-\lambda}).
    \end{equation}
Then,  there exists $t > 0$ such that,  uniformly for $\boldsymbol{f} \in \mathcal{E}_g$ and $|\boldsymbol{f} - \boldsymbol{f}_{\star}| < t$,  we have an expansion for the correlators:
\begin{equation}
\label{expco}W_{\frac{\overrightarrow{\boldsymbol{N}}}{N}}(x) =  NW_{\frac{\overrightarrow{\boldsymbol{N}}}{N}}^{\{-1\}}(x) +O(1).
\end{equation}
 \eqref{expco} holds uniformly for $x$ in compact regions outside $\boldsymbol{A}$(in our case in particular near the critical point $c_{0}$). $\boldsymbol{f}\rightarrow
 W_{\boldsymbol{f}}^{\{-1\}}(x)$ extends as a  smooth function  in a neighborhood of  $\boldsymbol{f}_{\star}$ for any $x\in \mathbb{C}\backslash \boldsymbol{S}$. 
\end{thm}

\noindent
{\bf Proof of Proposition \ref{order0}.}
By the partition function estimate from Theorem \ref{part},  for a sufficiently small $\kappa$, we have
\begin{equation}\label{G1}
\begin{aligned} 
&Z_{N-M, \beta}^{V; \boldsymbol{A}}=\sum_{N_{1}+\cdots+N_{g}=N-M}\frac{(N-M)!}{N_{1}!\cdots N_{g}!}\cdot Z^{V; \boldsymbol{A}}_{N-M, \frac{\overrightarrow{\boldsymbol{N}}}{N-M}, \beta}, \\
&\approx
 \sum_{ 
 \begin{subarray}
\quad N_{1}+\cdots+N_{g}=N-M, \\
 |\frac{\overrightarrow{\boldsymbol{N}}}{N-M}-\boldsymbol{f_{\star}}|<\kappa
 \end{subarray}
 }(N-M)^{(\frac{\beta}{2})(N-M)+e}\exp\left((N-M)^{2}F^{\{- 2\}}_{\frac{\overrightarrow{\boldsymbol{N}}}{N-M}, \beta}+(N-M)F^{\{-1\}}_{\frac{\overrightarrow{\boldsymbol{N}}}{N-M}, \beta}\right)\\
 &
+O(e^{-C_\kappa N^2})Z_{N-M, \beta}^{V; \boldsymbol{A}}\end{aligned}
\end{equation}
with some $C_\kappa>0$.
In the last step we applied the large deviation principle for the empirical measure $L_{N-M}$; in other words,  the sum over all $\overrightarrow{N}$ such that $|\frac{\overrightarrow{\boldsymbol{N}}}{N-M}-\boldsymbol{f}_{\star}|\geq\kappa$ divided by 
$Z_{N-M, \beta}^{V; \boldsymbol{A}}$
 is negligible since it is the probability that the filling fractions are away from the equilibrium ones, a set on which  the distance between the empirical measure $L_{N-M}$ and the equilibrium measure is positive. We used Theorem \ref{part} to estimate each partition functions in the remaining sum. 
Similarly, 
\begin{equation}\label{G2}
\begin{aligned}
 Z_{N, \beta}^{V; \boldsymbol{A}}&=\sum_{N_{1}+\cdots+N_{g}=N}\frac{N!}{N_{1}!\cdots N_{g}!}\cdot Z^{\frac{N}{N}; \boldsymbol{A}}_{N, \frac{\overrightarrow{\boldsymbol{N}}}{N}, \beta}, \\
&\approx \sum_{\begin{subarray}
\quad
N_{1}+\cdots+N_{g}=N, \\
|\frac{\overrightarrow{\boldsymbol{N}}}{N}-\boldsymbol{f_{\star}}|<\kappa\end{subarray}}(N)^{(\frac{\beta}{2})(N)+e}\exp\left((N)^{2}F^{\{-2\}}_{\frac{\overrightarrow{\boldsymbol{N}}}{N},\beta}+(N)F^{\{-1\}}_{\frac{\overrightarrow{\boldsymbol{N}}}{N},\beta}\right).
\end{aligned}
\end{equation}\\
All that is left to do is to analyze the limiting behavior of : 
\begin{equation}\label{G3}
L_{{K}}:= \sum_{\begin{subarray}
\quad
N_{1}+\cdots+N_{g}=K, \\
|\frac{\overrightarrow{\boldsymbol{N}}}{K}-\boldsymbol{f_{\star}}|<\kappa\end{subarray}} \exp((K)^{2}F^{\{-2\}}_{\frac{\overrightarrow{\boldsymbol{N}}}{K}, \beta}+(K)F^{\{-1\}}_{\frac{\overrightarrow{\boldsymbol{N}}}{K}, \beta}) .
\end{equation}
Here $\overrightarrow{\boldsymbol{N}}=(N_{1}, \cdots, N_{g})$ with $\sum {N_{i}}=K$.
{This is done in Lemma \ref{lemG4} below, from where the rest of the argument is exactly as in the one-cut case:}
\begin{lem}\label{lemG4}
\begin{equation}\label{gfd}
L_{{K}}\approx \exp(K^{2}F^{\{-2\}}_{\boldsymbol{f_{\star}},\beta}+KF^{\{-1\}}_{\boldsymbol{f_{\star}},\beta})
\end{equation}
\end{lem} 
\begin{proof}
According to Theorem \ref{part}, { for ${\boldsymbol{f}}$ sufficiently close to $\boldsymbol{f}_{\star}$, $F_{\boldsymbol{f}}^{\{-1\}}$ and $F_{\boldsymbol{f}}^{\{-2\}}$ are smooth,  and the Hessian of $F^{\{-2\}}_{\eee,\beta}$ is negative definite at $\boldsymbol{f_{\star}}$.} Thus we can find constant c, C such that for $|\boldsymbol{f}-\boldsymbol{f_{\star}}|\le \kappa <t$, we have 
\begin{equation}\label{lem1}|F^{\{-1\}}_{\boldsymbol{f}, \beta}-F^{\{-1\}}_{\boldsymbol{\boldsymbol{f_{\star}}}, \beta}|\leq C|\boldsymbol{f}-\boldsymbol{f_{\star}}|, \end{equation}
\begin{equation}\label{lem2}
F^{\{-2\}}_{\boldsymbol{f}, \beta}-F^{\{-2\}}_{\boldsymbol{\boldsymbol{f_{\star}}}, \beta}\leq-c|\boldsymbol{f}-\boldsymbol{f_{\star}}|^{2}, \end{equation}
 \begin{equation}\label{lem3}|F^{\{-2\}}_{\boldsymbol{f}, \beta}-F^{\{-2\}}_{\boldsymbol{\boldsymbol{f_{\star}}}, \beta}|\leq C|\boldsymbol{f}-\boldsymbol{f_{\star}}|^{2}.\end{equation}

{We first derive the lower bound of $L_{K}$. Indeed, there exists at least one $\overrightarrow{\boldsymbol{N_{1}}}:=(N_{1},\cdots, N_{g})$,  such that  $|\frac{\overrightarrow{\boldsymbol{N_{1}}}}{K}-\boldsymbol{f_{\star}}|{\lesssim}\frac{1}{K}$ and $N_{1}+\cdots+ N_{g}=K$. Thus
we can easily get the following lower bound of $L_{k}$ with (\ref{lem1}), (\ref{lem3}):}
\begin{equation} 
L_{{K}}{\geq \exp(K^{2}F^{\{-2\}}_{\frac{\overrightarrow{\boldsymbol{N}}_{1}}{K}, \beta}+KF^{\{-1\}}_{\frac{\overrightarrow{\boldsymbol{N}}_{1}}{K}, \beta})\gtrsim} \exp(K^{2}F^{\{-2\}}_{\boldsymbol{f_{\star}, \beta} }+KF^{\{-1\}}_{\boldsymbol{f_{\star}}, \beta}).
\end{equation}
Next,  we compute the upper bound of $L_{K}$.
Direct computation leads to
\begin{equation}\label{tempupper}
\begin{aligned}
&\frac{L_{K}}{\exp(K^{2}F^{\{-2\}}_{\boldsymbol{f_{\star}}, \beta}+KF^{\{-1\}}_{\boldsymbol{f_{\star}}, \beta})}\\
&=\sum_{\begin{subarray}
\quad
N_{1}+\cdots+N_{g}=K, \\
|\frac{\overrightarrow{\boldsymbol{N}}}{K}-\boldsymbol{f_{\star}}|<\kappa\end{subarray}} \exp(K^{2}(F^{\{-2\}}_{\frac{\overrightarrow{\boldsymbol{N}}}{K}, \beta}-F^{\{-2\}}_{\eee_{\star}, \beta})+K(F^{\{-1\}}_{\frac{\overrightarrow{\boldsymbol{K}}}{K}, \beta}-F^{\{-2\}}_{\eee_{\star}, \beta}))\\
&\leq  \sum_{|\frac{\overrightarrow{\boldsymbol{N}}}{K}-\boldsymbol{f_{\star}}|\le \kappa}\exp(-cK^{2}|\frac{\overrightarrow{\boldsymbol{N}}}{K}-\boldsymbol{f_{\star}}|^2+CK|\frac{\overrightarrow{\boldsymbol{N}}}{K}-\boldsymbol{f_{\star}}|).
\end{aligned}
\end{equation}
In the last step of \eqref{tempupper}, we use \eqref{lem1} and \eqref{lem2}. It is clear that the above right hand side is bounded from which the claim follows. 
\end{proof}
Having established Lemma \ref{lemG4},  we complete the proof of Proposition \ref{order0} as in the one-cut case.

To prove Proposition \ref{o2}, let us notice first that we can replace $\Xi$ 
by taking the integral only on filling fractions close to that of the equilibrium measure since the error 
will be otherwise of order $e^{-cN^2}$: for a given configuration $\boldsymbol{\lambda}=(\lambda_1,\ldots, \lambda_{N-M-1})$ denote
${\overrightarrow{\boldsymbol{N(\lambda)}}}= (N_1(\lambda),\ldots, N_g(\lambda))$ with $N_h(\lambda)=\#\{i: \lambda_i\in \boldsymbol{A}_h\}$.
Then hereafter we replace $\Xi$ by its localized version :
\begin{eqnarray}\label{defLY} 
{\Xi(\eta_1, \cdots, \eta_M):=\mu^{\frac{N}{N-M}V;\boldsymbol{A}}_{N-M, \beta}(\prod_{j=1}^{M}e^{\beta\sum_{i=1}^{N-M}\ln|\eta_{j}-\lambda_{i}|-\frac{\beta}{2}(N-M)V(\eta_{j})}\mathbf{1}_{|\frac{{\overrightarrow{\boldsymbol{N(\lambda)}}}}{N-M} -\eee_{\star}|<\kappa})}\,.
\end{eqnarray}
For any $\kappa>0$, the part we cut-off can be controlled by large deviation principle of the empirical measure $L_{N}$, which is of order $e^{-c_{\kappa}N^{2}}$. 
All our estimates on $F_{N,M}\Xi$ below will be made up to this error that we will not write done to simplify the exposition. The proof goes again through an expansion of terms where the filling fractions are fixed.

 \begin{equation}\label{keykey}
F_{N, M}\Xi (\eta_{1}, \cdots, \eta_{M}){=} 
\sum_{
\begin{subarray}
\quad N-M=N_{1}+\cdots+N_{g}, \\|\frac{\overrightarrow{\boldsymbol{N}}}{N-M}-\boldsymbol{f_{\star}}|<\kappa
\end{subarray}
}
c_{\overrightarrow{\boldsymbol{N}}}d_{\overrightarrow{\boldsymbol{N}}}, 
\end{equation}
where
\begin{itemize}
\item
\begin{equation} \label{def c}
c_{\overrightarrow{\boldsymbol{N}}}:=\frac{1}{Z_{N-M, \beta}^{V;\boldsymbol{A}}}
\frac{(N-M)!}{N_{1}!\cdots
N_{g}!}\cdot Z^{V; \boldsymbol{A}}_{N-M, \frac{\overrightarrow{\boldsymbol{N}}}{N-M}, \beta}, 
\end{equation}
\item
\begin{equation} \label{def d} d_{\overrightarrow{\boldsymbol{N}}}:=\mu^{V;\boldsymbol{A}}_{N-M, \frac{\overrightarrow{\boldsymbol{N}}}{N-M}\beta}(\prod_{i=1}^{M}e^{\sum_{j=1}^{N-M}(-\frac{\beta}{2}V(\eta_{i})+\beta \ln|\lambda_{j}-\eta_{i}|-\frac{\beta}{2}V(\lambda_{j}))}).
\end{equation}
\end{itemize}
We use the concentration 
Lemma \ref{central}  
with $$h(x)=\sum_{i=1}^M h_{\eta_i} (x)\,, \quad h_\eta(x)= 
-\frac{\beta}{2}V(x)+\beta \ln|\eta_i-x |\, .$$
Note that $\|h\|_\La^2$ is of order $M^2$ for {$\eta_{1},\cdots,\eta_M$} close to $c_0$ and $\|h\|_\infty$ is of order $M$. 
We first estimate $d_{\overrightarrow{\boldsymbol{N}}}$ and then substitute the estimate into (\ref{keykey}).

\begin{eqnarray}
d_{\overrightarrow{\boldsymbol{N}}}
&=&\mu^{V;\boldsymbol{A}}_{N-M, \frac{\overrightarrow{\boldsymbol{N}}}{N-M};\beta}(e^{\sum_{i=1}^{N-M} h(\lambda_i)})\prod_{j=1}^{M}e^{-\frac{\beta}{2}(N-M)V(\eta_{j})} \nonumber\\
&\leq &Ce^{CM^{2}} e^{(N-M) \mu_{\rm{eq}, \frac{\overrightarrow{\boldsymbol{N}}}{N-M}}(h)}\prod_{j=1}^{M}e^{-\frac{\beta}{2}(N-M)V(\eta_{j})}.\label{d}
\end{eqnarray}

Next,  we want to substitute  $\mu_{\rm eq, \frac{\overrightarrow{\boldsymbol{N}}}{N-M}}$ by $\mu_{\rm eq}$.
By  Appendix A.1 in \cite{BG2},  $\boldsymbol{f}\to \mu_{\rm eq,\boldsymbol{f}} (h_\eta)$ is Lipschitz in a neighborhood of $\boldsymbol{f_{\star}}$, uniformly in $\eta\in [c_0-\epsilon,c_0+\epsilon]$ so that 
\begin{equation}\label{k3}
|\mu_{\rm{eq}}(h)-
\mu_{\rm{eq}, \frac{\overrightarrow{\boldsymbol{N}}}{N-M}}(h)|=|\mu_{\rm{eq}, \boldsymbol{f_{\star}}}(h)-
\mu_{\rm{eq}, \frac{\overrightarrow{\boldsymbol{N}}}{N-M}}(h)|\leq CM|\frac{\overrightarrow{\boldsymbol{N}}}{N-M}-\boldsymbol{f_{\star}}|.
\end{equation}
Combining (\ref{d}), \eqref{keykey},  (\ref{k3}) gives 
\begin{equation}\label{k4}
\begin{aligned}
&F_{N, M}\Xi\left(\eta_{1}, \cdots, \eta_{M} \right)
\lesssim e^{CM^{2}}(\prod_{j=1}^{M}e^{-\frac{\beta}{2}NV(\eta_{j})})\\
&\times \sum_{|\frac{\overrightarrow{\boldsymbol{N}}}{N-M}-\boldsymbol{f_{\star}}|<\kappa}
c_{\overrightarrow{\boldsymbol{N}}}\exp\left((N-M)(\mu_{{\rm eq}}^{V;\boldsymbol{A}}(h)+CM|\frac{\overrightarrow{\boldsymbol{N}}}{N-M}-\boldsymbol{f_{\star}}|)+O(M)\right).
\end{aligned}
\end{equation} 
Finally, observe as in \eqref{tempupper}  that $c_{\overrightarrow{\boldsymbol{N}}}$ has a sub-Gaussian tail, that is:
\begin{equation}\label{k7}
c_{\overrightarrow{\boldsymbol{N}}}\leq Ce^{-c(N-M)^{2}|\frac{\overrightarrow{\boldsymbol{N}}}{N-M}-\boldsymbol{f_{\star}}|^{2}+C(N-M)|\frac{\overrightarrow{\boldsymbol{N}}}{N-M}-\boldsymbol{f_{\star}}|}.
\end{equation}
so that we deduce that
\begin{equation}\label{k6}
\sum_{|\frac{\overrightarrow{\boldsymbol{N}}}{N-M}-\boldsymbol{f_{\star}}|<\kappa}c_{\overrightarrow{\boldsymbol{N}}}e^{CM(N-M)|\frac{\overrightarrow{\boldsymbol{N}}}{N-M}-\boldsymbol{f_{\star}}|}\leq Ce^{CM^{2}}, 
\end{equation}
{Indeed, $$CM(N-M)|\frac{\overrightarrow{\boldsymbol{N}}}{N-M}-\boldsymbol{f_{\star}}|\leq 
{\frac{2}{c} C^2M^{2}+\frac{c}{2}|N-M|^{2}|\frac{\overrightarrow{\boldsymbol{N}}}{N-M}-\boldsymbol{f_{\star}}|^{2}}$$ so that   the term $c_{\overrightarrow{\boldsymbol{N}}}e^{CM(N-M)|\frac{\overrightarrow{\boldsymbol{N}}}{N-M}-\boldsymbol{f_{\star}}|}$ has also a sub-Gaussian tail up to multiplying  it by $e^{CM^{2}}$, thus \eqref{k6} follows from \eqref{k7}. }
Therefore \eqref{k4} yields the desired upper bound, as in \eqref{eqr}:

\begin{equation}\label{k5}
\begin{aligned}
F_{N.M}\Xi({\eta_{1}, \cdots, \eta_{M}})&\lesssim& e^{CM^{2}}(\prod_{j=1}^{M}e^{-\frac{\beta}{2}(N-M)V(\eta_{j})})e^{(N-M)\mu_{{\rm eq}}(h)} \\
&\le& Ce^{CM^2}e^{-\frac{\beta NM}{2}\int V(\eta)d\mu_{{\rm eq}}(\eta)}e^{-N\frac{\beta}{2}\sum_{j=1}^{M}{\mathcal{J}}^{V;\boldsymbol{B}}(\eta_{j})}\,.
\end{aligned}
\end{equation}
The proof of \eqref{lb} is similar to  the one cut case. From \eqref{keykey}, choosing $\overrightarrow{\boldsymbol{N}}$ so that
$|\overrightarrow{\boldsymbol{N}}- N\boldsymbol{f}_\star|\le 1$, we get
$$F_{N,1}\Xi(\eta)\ge c_{\overrightarrow{\boldsymbol{N}}} d_{\overrightarrow{\boldsymbol{N}}}\,.$$
We use Lemma \ref{central} to lower bound the term in $d_{\overrightarrow{\boldsymbol{N}}}$:
$$d_{\overrightarrow{\boldsymbol{N}}} \ge e^{-cM^2} e^{ -\frac{\beta}{2}N( \mathcal J^{V;\boldsymbol{B}}(\eta) +\int V(x)d\mu_{\rm eq}(x))}\,.$$
The $c_{\overrightarrow{\boldsymbol{N}}}$ term is bounded from below by  the partition function estimate from Theorem \ref{part} as well as the upper bound for
$ Z_{N, \beta}^{V; \boldsymbol{A}}$ provided by \eqref{G2} and Lemma \ref{lemG4}.

\section{Acknowledgments}
Part of this work was completed as part of 
the MIT 
SPUR program over the summer of 2013.  AG was partially supported by the Simons Foundation and by NSF Grant DMS-1307704. 
The authors are very grateful to a anonymous referee for his helpful comments.

\appendix

\section{Concentration lemmas}
\subsection{Moments  estimate}
Here we deduce the following lemma by using the estimate for correlator, by a direct application of Cauchy's integral formula. 
\begin{lem}\label{core}
Let Assumption \ref{main-assume} and \ref{as} hold,  let $\overrightarrow{\boldsymbol{N}}=(N_1,\ldots,N_g)$ so that $\sum N_i=N$ and 
let $\mu_{N, \frac{\overrightarrow{\boldsymbol{N}}}{N}, \beta}^{V;\boldsymbol{A}}$be the fixed filling fractions measure, and $\mu_{{\rm eq}, \frac{\overrightarrow{\boldsymbol{N}}}{N},}$ be its limiting measure. Let $h$ be a function that is holomorphic in a open neighborhood $\boldsymbol{U}$ of $\boldsymbol{A}$. Then, 
\begin{equation}
|\mu_{N, \frac{\overrightarrow{\boldsymbol{N}}}{N}, \beta}^{V;\boldsymbol{A}}(\sum h(\lambda_{i}))
-N\mu_{{\rm eq}, \frac{\overrightarrow{\boldsymbol{N}}}{N}}(h)|\lesssim C\|h\|_\infty.
\end{equation}
where $\|h\|_\infty$ is the supremum norm of $h$ on a contour around $\boldsymbol{A}$ inside $\boldsymbol{U}$.\\
This result in particular  hold in the one cut case where the fraction is always set  to $\overrightarrow{\boldsymbol{N}}=N$.
\end{lem}
\begin{proof}
 As $h$ is holomorphic, we can write by Cauchy formula, for a contour $\mathcal C$ around $\boldsymbol{A}$, 
 $$\mu_{N, \boldsymbol{f}, \beta}^{V;\boldsymbol{A}}(\sum h(\lambda_{i}))
 =\frac{1}{2i\pi}\int_{\mathcal C} h(\xi) W_{ \boldsymbol{f}}(\xi) d\xi$$
 from which the estimate follows from \eqref{expco} (and \eqref{expco0} in the one cut case).
 \end{proof}

\subsection{Concentration estimates}
We  assumed in Assumption \ref{as} that  $V$ is strictly convex in a neighborhood of $\boldsymbol{A}$ in order to use the 
following concentration inequality, see sections 2.3.2, 4.4.17, 4.4.26 in \cite{AGZ} for more details. Note here that we fix the component in which each
eigenvalue is living so that indeed they only see a convex potential.
\begin{lem}[Concentration Inequality]
\label{comi} Let $\boldsymbol{f}\in \mathcal E_g$ be given. Let $V$ be a smooth function such that $V''(x)\geq C>0$
for all $x\in\boldsymbol{A}$. Let $h$ be a function that is class
$C^{1}$ on $\mathbb{R}^{N}$. Let $\overrightarrow{\boldsymbol{N}}=(N_1,\ldots,N_g)$ so that $\sum N_i=N$. Then
\begin{equation}
\mu_{N,  \frac{\overrightarrow{\boldsymbol{N}}}{N},
\beta}^{V;\boldsymbol{A}}\left[\exp\{\left(f-\mu_{N,  \frac{\overrightarrow{\boldsymbol{N}}}{N}, \beta}^{V;\boldsymbol{A}}(h)\right)\}\right]\lesssim e^{\frac{1}{NC}\|h\|_{\mathcal{L}}^{2}}, 
\end{equation}
where
\[
\|h\|_{\mathcal{L}}:=\sqrt{\sum_{i=1}^{N}\sup_{x\in\boldsymbol{A}^N}|\partial_{\lambda_{i}}h(x)|^{2}}.
\]
\end{lem}
Note that this lemma applies in particular in the one cut case. 

\begin{lem}\label{central}
Let Assumption \ref{main-assume} and \ref{as} hold and let $\overrightarrow{\boldsymbol{N}}=(N_1,\ldots,N_g)$ so that $\sum N_i=N$. 
Then, there exists a finite constant $C$ so that for any  h holomorphic
in an open   neighborhood of $\boldsymbol{A}$, we have 
\begin{equation}\label{eq1}e^{-C\|h\|_\infty} \lesssim
\mu_{N,\frac{\overrightarrow{\boldsymbol{N}}}{N},
\beta}^{V;\boldsymbol{A}}(\exp(\sum_{i=1}^N ( h(\lambda_{i})-\int h(\eta)d\mu_{{\rm eq}, \frac{\overrightarrow{\boldsymbol{N}}}{N}}(\eta)))) \lesssim e^ {C(\|h\|_{\mathcal{L}}^{2}+\|h\|_\infty)}\,.
\end{equation}
\end{lem}
\begin{proof} 
By Jensen's Inequality:
\begin{equation}\label{eq3}
\begin{aligned}
\mu_{N, \frac{\overrightarrow{\boldsymbol{N}}}{N}, \beta}^{V;\boldsymbol{A}}(\exp(\sum h(\lambda_i)))
\geq \exp(\mu_{N, \frac{\overrightarrow{\boldsymbol{N}}}{N}, \beta}^{V;\boldsymbol{A}}(\sum h(\lambda_{i}))).
\end{aligned}
\end{equation}
The upper bound is based on the concentration equality for fixed filling fractions 
  of Lemma \ref{comi} with $\|\sum h(\lambda_i)\|_{\mathcal L}^2=N\|h\|_{\mathcal L}^2$
  so that
 \begin{equation}\label{eq4}
\begin{aligned}
\mu_{N, \frac{\overrightarrow{\boldsymbol{N}}}{N}, \beta}^{V;\boldsymbol{A}}(\exp(\sum h(\lambda_i)))
\leq \exp\{\frac{1}{C}\|h\|^2_{\mathcal L}\} \exp(\mu_{N, \frac{\overrightarrow{\boldsymbol{N}}}{N}, \beta}^{V;\boldsymbol{A}}(\sum h(\lambda_{i}))).
\end{aligned}
\end{equation}
Lemma \ref{comi} completes the proof.
\end{proof}

\bibliographystyle{alpha}
\bibliography{BG}
\end{document}